\documentclass{amsart}
\usepackage{amsfonts}
\usepackage{latexsym}
\usepackage{amssymb}
\usepackage{amsmath}
\usepackage{todonotes}
\usepackage{color}

\newcommand{\R}{\mathbb R}

\newtheorem{thm}{Theorem}[section]
\newtheorem{lemma}[thm]{Lemma}
\newtheorem{proposition}[thm]{Proposition}
\newtheorem{cor}[thm]{Corollary}

\theoremstyle{remark}
\newtheorem*{rmk}{Remark}

\begin{document}


\title{A reverse Rogers-Shephard inequality for log-concave functions}

\author{David Alonso-Guti\'{e}rrez}
\email{alonsod@unizar.es}
\address{{\sc David Alonso}: \'Area de An\'alisis Matem\'atico, Departamento de Matem\'aticas, Facultad de Ciencias, Universidad de Zaragoza, Pedro Cerbuna 12, 50009 Zaragoza (Spain), IUMA}
\thanks{Partially suppored by MINECO Project MTM2016-77710-P}
\date{\today}
\begin{abstract}
We will prove a reverse Rogers-Shephard inequality for log-concave functions.
In some particular cases, the method used for general log-concave functions can be slightly improved, allowing us to prove volume estimates for polars of $\ell_p$-diferences of convex bodies whose polar bodies under some condition on the barycenter of their polar bodies.
\end{abstract}
\maketitle
\section{Introduction and notation}

A convex body is a subset $K\subseteq\R^n$ which is convex, compact, and has non-empty interior. The Minkowski sum of two convex bodies $K,L\subseteq\R^n$ is the convex body defined by
$$
K+L:=\{x+y\in\R^n:x\in K, y\in L\}.
$$
Brunn-Minkowski inequality states that for any two convex bodies $K,L\subseteq\R^n$, if $|\cdot|$ denotes the volume (Lebesgue measure) of a measurable set in $\R^n$,
$$
|K+L|^\frac{1}{n}\geq|K|^\frac{1}{n}+|L|^\frac{1}{n},
$$
with equality if and only if $K$ and $L$ are homothetic. As a consequence, for any convex body $K\subseteq\R^n$, the difference body $K-K:=K+(-K)$ verifies that
$$
|K-K|\geq 2^n|K|,
$$
with equality if and only if $K$ is centrally symmetric. A reverse inequality was proved by Rogers and Shephard (see \cite{RS}). Rogers-Shephard inequality states that for any convex body $K\subseteq\R^n$,
$$
|K-K|\leq{2n\choose n}|K|,
$$
with equality if and only if $K$ is a simplex. This inequality was extended to any pair of convex bodies in \cite{RS2}, where it was proved that for any pair of convex bodies  $K,L\subseteq\R^n$ and any $x_0\in\R^n$
\begin{equation}\label{eq:RSTwoBodies}
|K\cap (x_0+L)||K-L|\leq{2n\choose n}|K||L|,
\end{equation}
with equality if and only if $K=L$ is a simplex (see \cite{AJV} for the characterization of equality). A reverse inequality was proved by Milman and Pajor \cite{MP}. For any pair of convex bodies  $K,L\subseteq\R^n$ with the same,
\begin{equation}\label{eq:MilmanPajor}
|K||L|\leq |K-L||K\cap L|.
\end{equation}

For any convex body $K\subseteq\R^n$ such that $0\in\textrm{int} K$, its polar body is the convex body defined by
$$
K^\circ:=\{x\in\R^n:\langle x,y\rangle\leq 1\,,\,\forall y\in K\},
$$
where $\langle\cdot,\cdot\rangle$ denotes the usual scalar product in $\R^n$.

It was proved in \cite{AGJV} (see also \cite{AEFO}) that for any two convex bodies $K,L\subseteq\R^n$ with $0\in \textrm{int}K\cap\textrm{int}L$
\begin{equation}\label{eq:VolumePolars}
|(K\cap L)^\circ||(K-L)^\circ|\leq |K^\circ||L^\circ|.
\end{equation}
This inequality strengthens another inequality proved by Rogers and Shephard, which states that for any two convex bodies $K,L\subseteq\R^n$, containing the origin
\begin{equation}\label{eq:RSConvexHull}
|K\cap L||\textrm{conv}\{K,-L\}|\leq2^n|K||L|
\end{equation}
and allowed to prove that there is equality in \eqref{eq:RSConvexHull} if and only if $K=L$ is a simplex with the origin in one of its vertices.

In \cite{F}, Firey proved a dual Brunn-Minkowski theorem. As a particular case, one obtains the following volume inequality for the volume of the polar of the difference body of a convex body, which improves \eqref{eq:VolumePolars} when $L=K$. For any convex body $K$ with $0\in \textrm{int} K$,
\begin{equation}\label{eq:RSpolar1Body}
|(K-K)^\circ|\leq \frac{1}{2^n}|K^\circ|.
\end{equation}
The aforementioned dual Brunn-Minkowski theorem was extended to other querma\ss integrals in \cite{F2}. The dual Brunn-Minkowski inequality  for querma\ss integrals was extended in \cite{HY} for $\ell_p$-sums of convex bodies (see definition of the $\ell_p$-sum $K+_pL$ of $K$ and $L$ below). As a particular case, the authors proved the following volume inequality for the volume of the polar of the $p$-difference body of a convex body. For any convex body $K$ with $0\in \textrm{int} K$,
\begin{equation}\label{eq:BMDualpSums}
|(K+_p(-K))^\circ|\leq \frac{1}{2^\frac{n}{p}}|K^\circ|.
\end{equation}

Rogers-Shephard inequality \eqref{eq:RSTwoBodies} was extended to the setting of log-concave function in \cite{AGJV}.  A function $f:\R^n\to\R$ is said to be log-concave if $f(x)=e^{-v(x)}$ with $v:\R^n\to (-\infty,\infty]$ a convex function. Given two log-concave functions $f,g:\R^n\to\R$, their Asplund product is defined by
$$
f\star g(x)=\sup_{x_1+x_2=x}f(x_1)g(x_2)=\sup_{y\in\R^n}f(y)g(x-y).
$$
and their convolution is defined by
$$
f*g(x)=\int_{\R^n}f(y)g(x-y)dy.
$$
 Notice that if $f=\chi_K$ and $g=\chi_L$ are the characteristic functions of two convex bodies then $f\star g(x)=\chi_{K+L}(x)$ and $f*g(x)=|K\cap (x-L)|$. The functional version of inequality \eqref{eq:RSTwoBodies} states that for any two integrable log-concave functions $f,g:\R^n\to\R$
\begin{equation}\label{eq:RSFunctional}
 \Vert f*g\Vert_\infty\int_{\R^n}f\star g(x)dx\leq{2n\choose n}\Vert f\Vert_\infty\Vert g\Vert_\infty\int_{\R^n}f(x)dx\int_{\R^n}g(x)dx,
\end{equation}
 with equality if and only if $\frac{f(x)}{\Vert f\Vert_\infty}=\frac{g(-x)}{\Vert g\Vert_\infty}$ is the characteristic function of an $n$-dimensional simplex.

 The purpose of this paper is to give a reverse inequality  to \eqref{eq:RSFunctional} in the same spirit as \eqref{eq:MilmanPajor} reverses Rogers-Shephard inequality.  Given an integrable log-concave function the entropy of $f$ is defined by
 $$
 \textrm{Ent}(f)=-\frac{\int_{\R^n}f(x)\log f(x)dx}{\int_{\R^n}f(x)dx}.
 $$
 We will prove
 \begin{thm}\label{thm:MilmanPajorLogConcaveProduct}
Let $f,g:\R^n\to\R$ be two integrable log-concave functions with opposite barycenter and such that $\Vert f\Vert_\infty=f(0)$ and $\Vert g\Vert_\infty=g(0)$. Then
$$
\Vert f\Vert_\infty\Vert g\Vert_\infty\int_{\R^{n}}f(x)dx\int_{\R^n}g(y)dy\leq e^{\left(1+\textrm{Ent}\left(\frac{f}{\Vert f\Vert_\infty}\right)+\textrm{Ent}\left(\frac{g}{\Vert g\Vert_\infty}\right)\right)}f*g(0)\int_{\R^n}f\star g(z)dz.
$$
\end{thm}

\begin{rmk}
Notice that if $f=\chi_K$ and $g=\chi_{-L}$, where $K$ and $L$ have the same barycenter we obtain
$$
|K||L|\leq e|K\cap L||K-L|,
$$
which is an inequality like \eqref{eq:MilmanPajor} with a slightly worse constant. If $f(x)=g(x)=e^{-h_K(x)}$ with $K$ a centered convex body we obtain
$$
|K^\circ|\leq e^{1+2n}|(K- K)^\circ|,
$$
with a constant of the order $c^n$, which is the right order as \eqref{eq:RSpolar1Body} shows. Nevertheless, for these particular cases we will obtain sharper constants.
\end{rmk}

Let us introduce some more notation. For any convex body $K$ with $0\in\textrm{int} K$ $h_K$ and $\Vert \cdot\Vert_K$ denote the support function and the Minkowski gauge associated to $K$, which are defined as
$$
h_K(x)=\max_{y\in K}\langle x,y\rangle,\hspace{1cm} \Vert x\Vert_K=\inf\{\lambda>0: x\in\lambda K\}.
$$
If $K$ is centrally symmetric $\Vert \cdot\Vert_K$ is a norm. Besides, for every $x\in\R^n$, $h_{K}(x)=\Vert x\Vert_{K^\circ}$ and $h_{K+L}(x)=h_K(x)+h_L(x)$ for any pair of convex bodies $K,L\subseteq\R^n$. The $\ell_p$-sum of the convex bodies $K$ and $L$, $K+_p L$, is the convex body defined by its support function
\begin{itemize}
\item$h_{K+_p L}(x):=\left(h_K^p(x)+h_L^p(x)\right)^\frac{1}{p}$
\item$h_{K+_\infty L}(x):=\max\{h_K(x),h_L(x)\}$.
\end{itemize}
Notice that $K+_1L=K+L$ and $K+_\infty L=\textrm{conv}\{K,L\}$. Besides, if $L=K$ we have that $K+_p K=2^{\frac{1}{p}}K$

The $\ell_p$-intersection of $K$ and $L$ is the convex body defined by
\begin{itemize}
\item$h_{K\cap_pL}(x):=\inf_{x_1+x_2=x}\left(h_{K}^p(x_1)+h_{L}^p(x_2)\right)^\frac{1}{p}$
\item$h_{K\cap_\infty L}(x):=\inf_{x_1+x_2=x}\max\{h_K(x_1),h_L(x_2)\}$.
\end{itemize}
Notice that $K\cap_1 L=K\cap L$ and $K\cap_\infty L=(K^\circ+L^\circ)^\circ$. Besides, if $L=K$ we have that $K\cap_p K=2^{\frac{1}{p}-1}K$

Following the idea of the proof of \eqref{eq:MilmanPajor} we will prove the following
\begin{thm}\label{thm: RSPolar}
Let $K, L\subseteq\R^n$ be convex bodies such that $K^\circ$ and $L^\circ$ have opposite barycenters. Then for any $p\geq 1$
$$
|(K\cap_p L)^\circ||(K+_p(-L))^\circ|\geq\frac{\Gamma\left(1+\frac{n}{p}\right)^2}{\Gamma\left(1+\frac{2n}{p}\right)}|K^\circ||L^\circ|.
$$
\end{thm}

In particular, taking $p=1$ we obtain the following reverse inequality to \eqref{eq:VolumePolars}.
\begin{cor}\label{cor: RSPolar}
Let $K, L\subseteq\R^n$ be convex bodies such that $K^\circ$ and $L^\circ$ have opposite barycenters. Then
$$
|(K\cap L)^\circ||(K-L)^\circ|\geq{2n\choose n}^{-1}|K^\circ||L^\circ|.
$$
In particular, taking $L=K$,
$$
|(K-K)^\circ|\geq{2n\choose n}^{-1}|K^\circ|.
$$
\end{cor}

Taking $p=\infty$ in Theorem \ref{thm: RSPolar} and taking into account that  $(K+_\infty (-L))^\circ=(\textrm{conv}\{K, -L\})^\circ=K^\circ\cap (-L^\circ)$ and $(K\cap_\infty L)^\circ=K^\circ+L^\circ$, changing $K$ by $K^\circ$ and $L$ by $-L^\circ$ we recover inequality \eqref{eq:MilmanPajor}.

\begin{rmk}
In \cite{HY}, it was shown that an inequality like the one in Corollary \ref{cor: RSPolar} cannot be obtained. The reason is that no restriction on the barycenter of $K$ or $K^\circ$ was imposed and then the volume of $K^\circ$ can be arbitrarily large.
\end{rmk}


The following reverse inequality will also be proved

\begin{thm}\label{thm: RSPolarReverse}
Let $K, L\subseteq\R^n$ be convex bodies with $0\in\textrm{int}K\cap\textrm{int}L$. Then for any $p\geq 1$
$$
|(K\cap_p L)^\circ||(K+_p(-L))^\circ|\leq{2n\choose n}\frac{\Gamma\left(1+\frac{n}{p}\right)^2}{\Gamma\left(1+\frac{2n}{p}\right)}|K^\circ||L^\circ|.
$$
\end{thm}
In particular, taking $p=1$ we recover inequality \eqref{eq:VolumePolars} and taking $p=\infty$ we recover Rogers-Shephard  inequality \eqref{eq:RSTwoBodies}.
In order to prove Theorem \ref{thm: RSPolar}, we will work in the context of log-concave functions for some particular functions, for which we can use a slightly different approach than in the proof of Theorem \ref{thm:MilmanPajorLogConcaveProduct}, working with level sets of log-concave functions instead of epigraph of their logarithms.

For any log-concave function $f$, its polar function is defined by
$$
f^\circ(x):=\inf_{y\in\R^n}\left(\frac{e^{-\langle x,y\rangle}}{f(y)}\right)=e^{-\mathcal{L}(-\log f)(x)},
$$
where $\mathcal{L}$ denotes the Legendre transform
$$
\mathcal{L}(u)(x)=\sup_{y\in\R^n}\langle x,y\rangle-u(y).
$$
Since the Legendre transform of a convex function is convex, $f^\circ$ is convex. Besides, for any log-concave upper semi-continuous function $f^{\circ\circ}=f$. Notice that for any $1<p<\infty$ if $f(x)=e^{-\frac{1}{p}\Vert x\Vert_K^p}$ then $f^\circ(x)=e^{-\frac{1}{q}\Vert x\Vert_{K^\circ}^q}$, where $\frac{1}{p}+\frac{1}{q}=1$ and $\chi_K^\circ(x)=e^{-\Vert x\Vert_{K^\circ}}$. Besides, if $f$ and $g$ are two log-concave functions then
$$
(f\star g)^\circ=f^\circ g^\circ.
$$
\begin{rmk}
Considering $f(x)=e^{-\frac{1}{p}h_K^p(x)}$ and $g(y)=e^{-\frac{1}{p}h_K^p(y)}$ we see that, on the one hand,
\begin{eqnarray*}
f\star g(x)&=&\sup_{x_1+x_2=x}f(x_1)g(x_2)=\sup_{x_1+x_2=x}e^{-\frac{1}{p}(h_{K}^p(x_1)+h_L^p(x_2))}\cr
&=& e^{-\frac{1}{p}\left(\inf_{x_1+x_2=x}h_{K}^p(x_1)+h_{L}^p(x_2)\right)}= e^{-\frac{1}{p}h_{(K\cap_p L)}^p(x)}.\cr
\end{eqnarray*}
 and, on the other hand
$$
f\star g(x)=(f^\circ g^\circ)^\circ(x)=\left(e^{-\frac{1}{q}h_{K^\circ+_{q}L^\circ}^q}\right)^\circ(x)=e^{-\frac{1}{p}h_{(K^\circ+_{q}L^\circ)^\circ}^q(x)}.
$$
and then $K\cap_p L=(K^\circ+_{q}L^\circ)^\circ$.
\end{rmk}

The paper is organized as follows. In Section \ref{sec:TechnicalLemmas} we will give the proof of some technical lemmas that will be needed in order to prove the results. In Section \ref{sec:MilmanPajorFunctionsGeneral} we will prove Theorem \ref{thm:MilmanPajorLogConcaveProduct} and in Section \ref{sec:ResultsEllpSums} we will prove Theorems \ref{thm: RSPolar} and \ref{thm: RSPolarReverse}.

\section{Technical lemmas}\label{sec:TechnicalLemmas}

In this section we will collect the technical lemmas that we will use to prove our results. The following lemma can be found in \cite {MP} (see also \cite[Lemma 4.1.21]{AGM}) and is crucial in the proof of \eqref{eq:MilmanPajor}. Since the proof of our results heavily rely on it, we reproduce the proof here for the sake of completeness.

\begin{lemma}\label{lem:Log-concaveAndProbability}
Let $\mu$ be a probability measure on $\R^n$ and let $\psi:\R^n\to\R$ be a non-negative log-concave function with finite, positive integral. Then
$$
\int_{\R^n}\psi(x)d\mu(x)\leq\psi\left(\int_{\R^n}x\frac{\psi(x)}{\int_{\R^n} \psi(y)d\mu(y)}d\mu(x)\right).
$$
\end{lemma}

\begin{proof}
The function  $f(t):(0,\infty)\to\R$ given by $f(t)=t\log t$ is convex. Thus, by Jensen's inequality,
$$
\int_{\R^n}\psi(x)\log\psi(x)d\mu(x)\geq\left(\int_{\R^n}\psi(x)d\mu(x)\right)\log\left(\int_{\R^n}\psi(x)d\mu(x)\right).
$$
Equivalently,
$$
\int_{\R^n}\log\psi(x)\frac{\psi(x)}{\int_{\R^n}\psi(y)d\mu(y)}d\mu(x)\geq\log\left(\int_{\R^n}\psi(x)d\mu(x)\right).
$$
Since $\log\psi$ is concave on $\{x\in\R^n:\psi(x)>0\}$ we have, using Jensen's inequality with respect to the probability measure $\frac{\psi}{\int_{\R^n}\psi(y)d\mu(y)}d\mu$, that
\begin{eqnarray*}
\log\left[\psi\left(\int_{\R^n}x\frac{\psi(x)}{\int_{\R^n}\psi(y)d\mu(y)}d\mu(x)\right)\right]&\geq&\int_{\R^n}\log\psi(x)\frac{\psi(x)}{\int_{\R^n}\psi(y)d\mu(y)}d\mu(x)\cr
&\geq&\log\left(\int_{\R^n}\psi(x)d\mu(x)\right),
\end{eqnarray*}
which is the assertion of the lemma.
\end{proof}

Given an integrable log-concave function $f$ with $f(0)>0$ and $p>0$, the set $K_p(f)$ is defined as
$$
K_p(f)=\left\{x\in\R^n\,:\,\int_0^\infty f(rx)r^{p-1}dr\geq\frac{f(0)}{p}\right\}.
$$
These bodies were first considered by K. Ball in \cite{B}, who also established their convexity. Notice that when $f=\chi_K$ is the characteristic function of a convex body, $\Vert x\Vert_{f}=\Vert x\Vert_K$. We will consider the particular case of $p=n$. Let $K_f$ be the unit ball of the norm given by
$$
\Vert x\Vert_{K_f}=\left(\frac{n}{f(0)}\int_0^\infty r^{n-1}f(rx)dr\right)^{-\frac{1}{n}}.
$$

The following well-known lemma relates the volume of $K_f$ and the integral of $f$.
\begin{lemma}\label{VolumeAndIntegral}
Let $f$ be an integrable log-concave function with $f(0)>0$. Then
$$
|K_f|=\frac{1}{f(0)}\int_{\R^n}f(x)dx.
$$
\end{lemma}
\begin{proof}
Integrating in polar coordinates we have that
\begin{eqnarray*}
|K_f|&=&\int_{K_f}dx=n|B_2^n|\int_{S^{n-1}}\int_0^{\rho_{K_f}(u)}r^{n-1}drd\sigma(u)\cr
&=&|B_2^n|\int_{S^{n-1}}\rho_{K_f}(u)^nd\sigma(u)\cr
&=&|B_2^n|\int_{S^{n-1}}\frac{n}{f(0)}\int_0^\infty r^{n-1}f(rx)drd\sigma(u)\cr
&=&\frac{1}{f(0)}\int_{\R^n}f(x)dx.
\end{eqnarray*}
\end{proof}

The following proposition gives an inclusion between the level sets of a log-concave function and its associated convex body $K_f$. We will denote, for any $t\in (0,1]$
$$
K_t=\{x\in\R^n: f(x)\geq t\Vert f\Vert_\infty\}.
$$
\begin{proposition}\label{K_fAndLevelSets}
Let $f\,:\,\R^n\to\R$ be an integrable log-concave function with $\Vert f\Vert_\infty=f(0)$. Then, for any $t\in(0,1]$ we have that
$$
t^\frac{1}{n}K_t\subseteq K_f.
$$
\end{proposition}
\begin{proof}
Let $t\in(0,1]$. Then for any $u\in S^{n-1}$
\begin{eqnarray*}
\rho_{K_f}(u)^n&=&\frac{n}{\Vert f\Vert_\infty}\int_0^\infty r^{n-1}f(ru)dr\cr
&\geq&\frac{n}{\Vert f\Vert_\infty}\int_0^{\rho_{K_t}(u)} r^{n-1}f(ru)dr\cr
&\geq&nt\int_0^{\rho_{K_t}(u)} r^{n-1}dr\cr
&=&t\rho_{K_t}(u)^n.
\end{eqnarray*}
Consequently
$$
\rho_{K_f}(u)\geq t^\frac{1}{n}\rho_{K_t}(u)
$$
and so
$$
K_f\supseteq t^\frac{1}{n}K_t.
$$
\end{proof}

Given an integrable log-concave function $f$ we compute the barycenter of the epigraph of the convex function $-\log \frac{f}{\Vert f\Vert_\infty}$ with respect to the measure with density $e^{-t}$.
\begin{lemma}\label{prop:Entropy}
Let $f:\R^n\to\R$ be an integrable log-concave function and let
$$
L:=\{(x,t)\in\R^n\times[0,\infty)\,:\,f(x)\geq e^{-t}\Vert f\Vert_{\infty}\}.
$$
Then
$$
\frac{\int_{L}xe^{-t}dtdx}{\int_{L}e^{-t}dtdx}=\frac{\int_{\R^n}x\frac{f(x)}{\Vert f\Vert_\infty}dx}{\int_{\R^n}\frac{f(x)}{\Vert f\Vert_\infty}dx}.
$$
$$
\frac{\int_{L}te^{-t}dtdx}{\int_{L}e^{-t}dtdx}=1+\textrm{Ent}\left(\frac{f}{\Vert f\Vert_\infty}\right).
$$
\end{lemma}
\begin{proof}
Notice that
\begin{eqnarray*}
\int_{L}e^{-t}dtdx&=&\int_0^\infty e^{-t}|\{x\in\R^n: f(x)\geq e^{-t}\Vert f\Vert_\infty\}|dt\cr
&=&\int_0^1 |\{x\in\R^n: f(x)\geq s\Vert f\Vert_\infty\}|ds\cr
&=&\int_{\R^n}\int_0^{\frac{f(x)}{\Vert f\Vert_\infty}}dsdx=\int_{\R^n}\frac{f(x)}{\Vert f\Vert_\infty}dx.
\end{eqnarray*}
Similarly
\begin{eqnarray*}
\int_{L}xe^{-t}dtdx&=&\int_0^\infty e^{-t}\int_{\{x\in\R^n: f(x)\geq e^{-t}\Vert f\Vert_\infty\}}xdxdt\cr
&=&\int_0^1 \int_{\{x\in\R^n: f(x)\geq s\Vert f\Vert_\infty\}}xdxds\cr
&=&\int_{\R^n}x\int_0^{\frac{f(x)}{\Vert f\Vert_\infty}}dsdx=\int_{\R^n}x\frac{f(x)}{\Vert f\Vert_\infty}dx,
\end{eqnarray*}
and
\begin{eqnarray*}
\int_{L}te^{-t}dtdx&=&\int_0^\infty te^{-t}|\{x\in\R^n: f(x)\geq e^{-t}\Vert f\Vert_\infty\}|dt\cr
&=&\int_0^1 (-\log s)|\{x\in\R^n: f(x)\geq s\Vert f\Vert_\infty\}|ds\cr
&=&\int_{\R^n}\int_0^{\frac{f(x)}{\Vert f\Vert_\infty}}(-\log s)dsdx\cr
&=&\int_{\R^n}\frac{f(x)}{\Vert f\Vert_\infty}dx-\int_{\R^n}\frac{f(x)}{\Vert f\Vert_\infty}\log\frac{f(x)}{\Vert f\Vert_\infty}dx.
\end{eqnarray*}
\end{proof}

\begin{lemma}\label{lem:VolumeConvexBody}
Let $K\subseteq\R^n$ be a convex body such that $0\in\textrm{int} K$ and $p\geq 1$. Then
$$
\int_{\R^n}e^{-h_{K}^p(x)}dx= \Gamma\left(1+\frac{n}{p}\right)|K^\circ|.
$$
\end{lemma}
\begin{proof}
\begin{eqnarray*}
\int_{\R^n}e^{-h_{K}^p(x)}dx&=&\int_{\R^n}\int_{h_{K}^p(x)}^\infty e^{-t}dtdx=\int_0^\infty |t^\frac{1}{p}K^\circ|e^{-t}|dt\cr
&=&|K^\circ|\int_0^\infty t^{\frac{n}{p}}e^{-t}dt=\Gamma\left(1+\frac{n}{p}\right)|K^\circ|.
\end{eqnarray*}
\end{proof}

\begin{lemma}\label{lem:CenteredLevelSets}
Let $p\geq 1$ and let $K,L\subseteq\R^n$ be convex bodies with $0\in\textrm{int}(K\cap L)$ such that $K^\circ$ and $L^\circ$ have opposite barycenter. Define $f(x)=e^{-h_{K}^p(x)}, g(y)=e^{-h_L^p(y)}$. Then for every $t\in (0,1]$ the level sets
$$
K_t:=\{x\in\R^{n}: f(x)\geq t\}=(-\log t)^\frac{1}{p}K^\circ,
$$
$$
\widetilde{K}_t:=\{y\in\R^{n}: g(y)\geq t\}=(-\log t)^\frac{1}{p}L^\circ,
$$
and
$$
\frac{1}{|\widetilde{L}_t|}\int_{K_t}x\left|\widetilde{K}_{\frac{t}{f(x)}}\right|dx+
\frac{1}{|\widetilde{L}_t|}\int_{\widetilde{K}_t}y\left|\widetilde{K}_{\frac{t}{g(y)}}\right|dy=0,
$$
where $\widetilde{L}_t=\{(x,y)\in\R^{2n}:f(x)g(-y)\geq t\}$.
\end{lemma}
\begin{proof}
Notice that by the definition of $K_t$ we have that for every $t\in (0,1]$
\begin{eqnarray*}
K_t:&=&\{x\in\R^{n}: f(x)\geq t\}=\{x\in\R^{n}: e^{-h_{K}^p(x)}\geq t\}\cr
&=&\{x\in\R^{n}: h_K(x)\leq (-\log t)^\frac{1}{p}\}=\{x\in\R^{n}: \Vert x\Vert_{K^\circ}\leq (-\log t)^\frac{1}{p}\}\cr
&=&(-\log t)^\frac{1}{p}K^\circ.
\end{eqnarray*}
This also gives the identity for $\widetilde{K}_t$. Then for every $x\in K_t$
\begin{eqnarray*}
\left|\widetilde{K}_{\frac{t}{f(x)}}\right|&=&\left|\widetilde{K}_{te^{\Vert x\Vert_{K^\circ}^p}}\right|=\left|(-\log t-\Vert x\Vert_{K^\circ}^p)^\frac{1}{p}L^\circ\right|\cr
&=&(-\log t-\Vert x\Vert_{K^\circ}^p)^\frac{n}{p}\left|L^\circ\right|.
\end{eqnarray*}
Thus
\begin{eqnarray*}
|\widetilde{L}_t|&=&\int_{K_t}\left|\widetilde{K}_{\frac{t}{f(x)}}\right|dx\cr
&=&|L^\circ|\int_{(-\log t)^\frac{1}{p}K^\circ}(-\log t-\Vert x\Vert_{K^\circ}^p)^\frac{n}{p}dx\cr
&=&(-\log t)^\frac{2n}{p}|L^\circ|\int_{K^\circ}(1-\Vert y\Vert_{K^\circ}^p)^\frac{n}{p}dx\cr
&=&(-\log t)^\frac{2n}{p}|L^\circ|\int_{K^\circ} \int_{\Vert y\Vert_{K^\circ}^p}^1\frac{n}{p}(1-s)^{\frac{n}{p}-1}dsdy\cr
&=&\frac{n}{p}(-\log t)^\frac{2n}{p}|L^\circ|\int_0^1(1-s)^{\frac{n}{p}-1}\int_{s^\frac{1}{p}K^\circ}dyds\cr
&=&\frac{n}{p}(-\log t)^\frac{2n}{p}|K^\circ||L^\circ|\int_0^1s^\frac{n}{p}(1-s)^{\frac{n}{p}-1}ds\cr
\end{eqnarray*}
and
\begin{eqnarray*}
\int_{K_t}x\left|\widetilde{K}_{\frac{t}{f(x)}}\right|dx&=&|L^\circ|\int_{(-\log t)^\frac{1}{p}K^\circ}x(-\log t-\Vert x\Vert_{K^\circ}^p)^\frac{n}{p}dx\cr
&=&(-\log t)^\frac{2n+1}{p}|L^\circ|\int_{K^\circ} y(1-\Vert y\Vert_{K^\circ}^p)^\frac{n}{p}dy\cr
&=&(-\log t)^\frac{2n+1}{p}|L^\circ|\int_{K^\circ} \int_{\Vert y\Vert_{K^\circ}^p}^1y\frac{n}{p}(1-s)^{\frac{n}{p}-1}dsdy\cr
&=&\frac{n}{p}(-\log t)^\frac{2n+1}{p}|L^\circ|\int_0^1(1-s)^{\frac{n}{p}-1}\int_{s^\frac{1}{p}K^\circ}ydyds\cr
&=&\frac{n}{p}(-\log t)^\frac{2n+1}{p}|L^\circ|\int_0^1s^\frac{n+1}{p}(1-s)^{\frac{n}{p}-1}ds\int_{K^\circ}ydy.\cr
\end{eqnarray*}
Consequently,
$$
\frac{1}{|\widetilde{L}_t|}\int_{K_t}x\left|\widetilde{K}_{\frac{t}{f(x)}}\right|dx=(-\log t)^\frac{1}{p}\frac{\beta\left(1+\frac{n+1}{p},\frac{n}{p}\right)}{\beta\left(1+\frac{n}{p},\frac{n}{p}\right)}\frac{1}{|K^\circ|}\int_{K^\circ}ydy,
$$
and, similarly
$$
\frac{1}{|\widetilde{L}_t|}\int_{\widetilde{K}_t}y\left|\widetilde{K}_{\frac{t}{g(y)}}\right|dy=(-\log t)^\frac{1}{p}\frac{\beta\left(1+\frac{n+1}{p},\frac{n}{p}\right)}{\beta\left(1+\frac{n}{p},\frac{n}{p}\right)}\frac{1}{|L^\circ|}\int_{L^\circ}ydy.
$$
Since $K^\circ$ and $L^\circ$ have opposite barycenter we obtain the result.
\end{proof}

\begin{lemma}\label{lem:Operations}
Let $p\geq 1$ and let $K,L\subseteq\R^n$ be convex bodies such that $0\in\textrm{int}K\cap\textrm{int}L$. Let $f(x)=e^{-h_{K}^p(x)}$, $g(y)=e^{-h_{L}^p(y)}$. Then for every $x\in\R^n$
\begin{itemize}
\item$f\star g(x)= e^{-h_{(K\cap_p L)}^p(x)}$,
\item$f(x)g(-x)= e^{-h_{K+_p(-L)}^p(x)}.$
\end{itemize}
\begin{proof}
Both identities follow from the definitions. On the one hand, from the definition of the Asplund product
\begin{eqnarray*}
\sup_{x_1+x_2=x}f(x_1)g(x_2)&=&\sup_{x_1+x_2=x}e^{-h_{K}^p(x_1)-h_L^p(x_2)}= e^{-\left(\inf_{x_1+x_2=x}h_{K}^p(x_1)+h_{K}^p(x_2)\right)}\cr
&=& e^{-h_{(K\cap_p L)}^p(x)}.\cr
\end{eqnarray*}
On the other hand, $$h_{K+_p(-L)}^p(x)=h_K^p(x)+h_{-L}^p(x)=h_K^p(x)+h_{L}^p(-x).$$
\end{proof}
\end{lemma}

\section{Reverse functional Rogers-Shephard inequality}\label{sec:MilmanPajorFunctionsGeneral}

In this section we will prove Theorem \ref{thm:MilmanPajorLogConcaveProduct}.

\begin{proof}[Proof of Theorem \ref{thm:MilmanPajorLogConcaveProduct}]
Let $F:\R^{2n}\to\R$ be the log-concave function $F(u,v)=f\left(\frac{u+v}{\sqrt{2}}\right)g\left(\frac{v-u}{\sqrt{2}}\right)$ and let $L\subseteq\R^{2n+1}$ be the convex set
$$
L=\{(u,v,t)\in\R^{2n+1}:F(u,v)\geq e^{-t}\Vert f\Vert_\infty\Vert g\Vert_\infty\}.
$$
We will call $H=\textrm{span}\{e_{n+1},\dots,e_{2n}, e_{2n+1}\}$ and let $M=P_{H}(L)$. Taking into account that for any $u,v\in\R^n$ $\frac{u+v}{\sqrt{2}}+\frac{v-u}{\sqrt{2}}=\sqrt{2}v$ we have that
\begin{eqnarray*}
M&=&\{(v,t)\in\R^{n+1}:\sup_{u\in\R^n}F(u,v)\geq e^{-t}\Vert f\Vert_\infty\Vert g\Vert_\infty\}\cr
&=&\{(v,t)\in\R^{n+1}:f\star g(\sqrt{2}v)\geq e^{-t}\Vert f\Vert_\infty\Vert g\Vert_\infty\}.\cr
\end{eqnarray*}
Notice that, by Brunn-Minkowski inequality, the function $\psi(v,t)$ supported on $M$ and defined by
\begin{eqnarray*}
\psi(v,t)&=&|L\cap((0,v,t)+H^\perp)|\cr
&=&|\{u\in\R^n:F(u,v)\geq e^{-t}\Vert f\Vert_\infty\Vert g\Vert_\infty\}|
\end{eqnarray*}
is log-concave. Besides, calling $\mu$ the probability measure on $M$ defined by
$$
d\mu(v,t)=\frac{\chi_M(v,t)e^{-t}dtdv}{\int_M e^{-t}dtdv}
$$
we have, changing variables $x=\frac{u+v}{\sqrt{2}}$ and $y=\frac{u-v}{\sqrt{2}}$, that
\begin{eqnarray*}
\int_{M}\psi(v,t)d\mu(v,t)&=&\int_{M}e^{-t}\psi(v,t)dtdv=\frac{\int_L e^{-t}dtdvdu}{\int_M e^{-t}dtdv}\cr
&=&\frac{\int_0^\infty e^{-t}\left|\left\{(u,v)\in\R^{2n}:f\left(\frac{u+v}{\sqrt{2}}\right)g\left(\frac{v-u}{\sqrt{2}}\right)\geq e^{-t}\Vert f\Vert_\infty \vert g\Vert_\infty\right\}\right|}{{\int_M e^{-t}dtdv}}\cr
&=&\frac{\int_{\R^{2n}}f\left(\frac{u+v}{\sqrt{2}}\right)g\left(\frac{v-u}{\sqrt{2}}\right)dvdu}{\int_{\R^n}f\star g(\sqrt{2}z)dz}\cr
&=&\frac{(\sqrt{2})^n\int_{\R^{n}}f(x)dx\int_{\R^n}g(y)dy}{\int_{\R^n}f\star g(z)dz}.
\end{eqnarray*}
Taking into account that $f$ and $g$ have opposite barycenter, changing again variables $x=\frac{u+v}{\sqrt{2}}$ and $y=\frac{u-v}{\sqrt{2}}$ we have
\begin{eqnarray*}
&&\frac{\int_M(v,t)\psi(v,t)d\mu(v,t)}{\int_M\psi(v,t)d\mu(v,t)}=\frac{\int_M(v,t)e^{-t}\psi(v,t)dtdv}{\int_M e^{-t}\psi(v,t)dvdt}\cr
&=&\frac{\int_L(v,t)e^{-t}dtdudv}{\int_L e^{-t}dtdudv}=\frac{\int_0^\infty \int_{\left\{(u,v)\in\R^{2n}:f\left(\frac{u+v}{\sqrt{2}}\right)g\left(\frac{v-u}{\sqrt{2}}\right)\geq e^{-t}\Vert f\Vert_\infty \vert g\Vert_\infty\right\}}(v,t)e^{-t}dtdudv}{\int_{\R^n} \frac{f(x)}{\Vert f\Vert_\infty}dx\int_{\R^n} \frac{g(y)}{\Vert g\Vert_\infty}dy}\cr
&=&\left(0,1+\textrm{Ent}\left(\frac{f}{\Vert f\Vert_\infty}\right)+\textrm{Ent}\left(\frac{g}{\Vert g\Vert_\infty}\right)\right).
\end{eqnarray*}
Consequently, by Lemma \ref{lem:Log-concaveAndProbability} we have that
$$
\frac{(\sqrt{2})^n\int_{\R^{n}}f(x)dx\int_{\R^n}g(y)dy}{\int_{\R^n}f\star g(z)dz}
$$
is bounded above by the volume of
$$
\left\{u\in\R^n: f\left(\frac{u}{\sqrt2}\right)g\left(\frac{-u}{\sqrt2}\right)\geq e^{-\left(1+\textrm{Ent}\left(\frac{f}{\Vert f\Vert_\infty}\right)+\textrm{Ent}\left(\frac{g}{\Vert g\Vert_\infty}\right)\right)}\Vert f\Vert_\infty\Vert g\Vert_\infty\right\}.
$$
Since $\Vert f\Vert_\infty=f(0)$ and $\Vert g\Vert_\infty=g(0)$, we have, by Lemma \ref{K_fAndLevelSets}, that the latter set is contained in
$$
e^{\frac{1}{n}\left(1+\textrm{Ent}\left(\frac{f}{\Vert f\Vert_\infty}\right)+\textrm{Ent}\left(\frac{g}{\Vert g\Vert_\infty}\right)\right)}K_{f g_{-}\left(\frac{\cdot}{\sqrt2}\right)},
$$
where $g_{-}(x)=g(-x)$. Taking volumes and using Lemma \ref{VolumeAndIntegral} we obtain
\begin{eqnarray*}
\frac{(\sqrt{2})^n\int_{\R^{n}}f(x)dx\int_{\R^n}g(y)dy}{\int_{\R^n}f\star g(z)dz}\leq e^{\left(1+\textrm{Ent}\left(\frac{f}{\Vert f\Vert_\infty}\right)+\textrm{Ent}\left(\frac{g}{\Vert g\Vert_\infty}\right)\right)}\int_{\R^n}\frac{f\left(\frac{u}{\sqrt2}\right)g\left(\frac{-u}{\sqrt2}\right)}{\Vert f\Vert_\infty\Vert g\Vert_\infty}du.
\end{eqnarray*}
Thus,
\begin{eqnarray*}
&&\Vert f\Vert_\infty\Vert g\Vert_\infty\int_{\R^{n}}f(x)dx\int_{\R^n}g(y)dy\leq e^{\left(1+\textrm{Ent}\left(\frac{f}{\Vert f\Vert_\infty}\right)+\textrm{Ent}\left(\frac{g}{\Vert g\Vert_\infty}\right)\right)}\cr
&\times&\int_{\R^n}f\star g(z)dz\int_{\R^n}f\left(u\right)g(-u)du,\cr
\end{eqnarray*}
which completes the proof.
\end{proof}

\section{Volume estimates for polars of $\ell_p$-differences of convex bodies}\label{sec:ResultsEllpSums}

In this section we will prove Theorems \ref{thm: RSPolar} and \ref{thm: RSPolarReverse}. The proof of Theorem \ref{thm: RSPolar} follows the same lines as the one in the previous section  for general log-concave functions. The main difference lies in the fact that we will consider functions with homothetic level sets and then we will have that not only our functions will be centered but every level set will be centered. This will allow us to work with every level set separately instead of with the whole epigraph.
\begin{proof}[Proof of Theorem \ref{thm: RSPolar}]
Let us consider the log-concave functions $f:\R^n\to[0,\infty)$ given by $f(x)=e^{-h_{K}^p(x)}$, $g(y)={e^{-h_L^p(y)}}$. Notice that $\Vert f\Vert_\infty=f(0)=1$, $\Vert g\Vert_\infty=g(0)=1$ and that, by Lemma \ref{lem:VolumeConvexBody},  $\Vert f\Vert_1=\Gamma\left(1+\frac{n}{p}\right)|K^\circ|$ and $\Vert g\Vert_1=\Gamma\left(1+\frac{n}{p}\right)|L^\circ|.$ Let $F:\R^{2n}\to\R$ be the function given by
$$
F(u,v)=f\left(\frac{u+v}{\sqrt{2}}\right)g\left(\frac{v-u}{\sqrt{2}}\right).
$$
Observe that $\Vert F\Vert_\infty=F(0,0)=1$ and that, changing variables $x=\frac{u+v}{\sqrt2}$, $y=\frac{u-v}{\sqrt2}$
\begin{eqnarray*}
\int_{\R^{2n}}F(u,v)dvdu&=&\int_{\R^{2n}}f\left(\frac{u+v}{\sqrt{2}}\right)g\left(\frac{v-u}{\sqrt{2}}\right)dvdu=\int_{\R^{2n}}f(x)g(-y)dydx\cr
&=&\Gamma\left(1+\frac{n}{p}\right)^2|K^\circ||L^\circ|.
\end{eqnarray*}
Let us call, for every $t\in(0,1]$,
$$
L_t:=\{(u,v)\in\R^{2n}:F(u,v)\geq t\}
$$
and let $M_t$ be the projection of $L_t$ onto the subspace $H=\textrm{span}\{e_{n+1},\dots, e_{2n}\}$, which we identify with $\R^n$. Thus,
$$
M_t=\{v\in\R^n:\max_{u\in\R^n}F(u,v)\geq t\}.
$$
Then, the function $\psi_t(v):=|L_t\cap(0,v)+H^\perp|$ is log-concave and has support $M_t$. Calling $\mu_t$ the uniform probability measure on $M_t$ we have that
\begin{eqnarray*}
\int_{M_t}\psi_t(v)d\mu_t(v)=\frac{1}{|M_t|}\int_{M_t}\psi_t(v)dv=\frac{|L_t|}{|M_t|},
\end{eqnarray*}
and, changing variables $x=\frac{u+v}{\sqrt2}$, $y=\frac{u-v}{\sqrt2}$,
\begin{eqnarray*}
\int_{M_t}v\frac{\psi_t(v)}{\int_{M_t}\psi(w)d\mu_t(w)}d\mu_t(v)&=&\frac{1}{|L_t|}\int_{M_t}v\psi_t(v)dv=\frac{1}{|L_t|}\int_{L_t}vdv\cr
&=&\frac{1}{|L_t|}\int_{\widetilde{L}_t}\frac{x-y}{\sqrt2}dydx,\cr
\end{eqnarray*}
where $\widetilde{L}_t$ is the level set $\widetilde{L}_t:=\{(x,y)\in\R^{2n}:f(x)g(-y)\geq t\}.$
Since, by Lemma \ref{lem:CenteredLevelSets},
\begin{eqnarray*}
\int_{\widetilde{L}_t}(x-y)dydx&=&\int_{\widetilde{L}_t}xdydx-\int_{\widetilde{L}_t}ydydx\cr
&=&\int_{K_t}x|\widetilde{K}_{\frac{t}{f(x)}}|dx-\int_{-\widetilde{K}_t}y|K_{\frac{t}{g(-y)}}|\cr
&=&\int_{K_t}x|\widetilde{K}_{\frac{t}{f(x)}}|dx+\int_{\widetilde{K}_t}y|K_{\frac{t}{g(y)}}|\cr
&=&0,
\end{eqnarray*}
we have that for every $t\in(0,1]$
$$
\int_{M_t}v\frac{\psi_t(v)}{\int_{M_t}\psi(w)d\mu_t(w)}d\mu_t(v)=0.
$$
Consequently, by Lemma \ref{lem:Log-concaveAndProbability}, for every $t\in(0,1]$
$$
\frac{|L_t|}{|M_t|}\leq\psi_t(0)=|L_t\cap H^{\perp}|=\left|\left\{u\in\R^n:f\left(\frac{u}{\sqrt2}\right)g\left(-\frac{u}{\sqrt2}\right)\geq t\right\}\right|.
$$
Equivalently, the volume of $L_t$ is bounded above by
$$
\left|\left\{v\in\R^n:\max_{u\in\R^n}f\left(\frac{u+v}{\sqrt{2}}\right)g\left(\frac{v-u}{\sqrt{2}}\right)\geq t\right\}\right|\left|\left\{u\in\R^n:f\left(\frac{u}{\sqrt2}\right)g\left(-\frac{u}{\sqrt2}\right)\geq t\right\}\right|.
$$
Integrating in $t\in(0,1]$ and using Fubini's theorem we obtain that
\begin{align*}
&\Gamma\left(1+\frac{n}{p}\right)^2|K^\circ||L^\circ|=\int_{\R^{2n}}F(u,v)dvdu\cr
&\leq\int_{\R^{2n}}\min\left\{\max_{\bar{u}\in\R^n}f\left(\frac{\bar{u}+v}{\sqrt{2}}\right)g\left(\frac{v-\bar{u}}{\sqrt{2}}\right),f\left(\frac{u}{\sqrt2}\right)g\left(-\frac{u}{\sqrt2}\right)\right\}dudv.\cr
\end{align*}
Notice that for every $\bar{u},v\in\R^n$ $\frac{\bar{u}+v}{\sqrt{2}}+\frac{v-\bar{u}}{\sqrt{2}}=\sqrt{2}v$ and then for every $v\in\R^{2n}$
$$
\max_{\bar{u}\in\R^n}f\left(\frac{\bar{u}+v}{\sqrt{2}}\right)g\left(\frac{v-\bar{u}}{\sqrt{2}}\right)=\max_{x_1+x_2=\sqrt{2}v}f(x_1)g(x_2)=f\star g(\sqrt{2}v).
$$
Thus, by Lemma \ref{lem:Operations} and Lemma \ref{lem:VolumeConvexBody}
\begin{eqnarray*}
\Gamma\left(1+\frac{n}{p}\right)^2|K^\circ||L^\circ|&\leq&\int_{\R^{2n}}\min\{f\star g(v), f(u)g(-u)\}dvdu\cr
&=&\int_{\R^{2n}}e^{-\max\left\{\Vert u\Vert_{(K+_p(-L))^\circ}^p, \Vert v\Vert_{(K\cap_p L)^\circ}^p\right\}}dvdu\cr
&=&\int_{\R^{2n}}e^{\Vert (u,v)\Vert^p_{((K\cap_p L))^\circ\times (K+_p(-L))^\circ}}dvdu\cr
&=&\Gamma\left(1+\frac{2n}{p}\right)|(K\cap_p L)^{\circ}||(K+_p(-L))^\circ|,
\end{eqnarray*}
which finishes the proof.
\end{proof}
The proof of Theorem \ref{thm: RSPolarReverse} follows the same idea, but now we use the inequality, due to Rogers and Shephard,
$$
|L_t|\geq{2n\choose n}^{-1}|P_{H}L_t||L_t\cap H^\perp|=|M_t||L_t\cap H^\perp|.
$$


\begin{thebibliography}{99?}



\bibitem[AGJV]{AGJV} {\sc D. Alonso-Guti\'{e}rrez, B. Gonz\'alez Merino, C. H. Jim\'enez, R. Villa}, \textit{Rogers-Shephard inequality for log-concave functions}, J. Func. Anal. {\bf 271 (11)} (2016), pp. 3269--3299.

\bibitem[AJV]{AJV}{\sc Alonso-Guti\'errez D., Jim\'enez C.H., Villa R.}
\textit{Brunn-Minkowski and Zhang inequalities for convolution bodies.}
Adv. in Math. {\bf 238} (2013), pp. 50--69.

\bibitem[AEFO]{AEFO}{\sc S. Artstein-Avidan, K. Einhorn, D. I. Florentin, Y. Ostrover}, \textit{On Godbersen's conjecture}, Geom. Dedicata {\bf 178 (1)} (2015), pp. 337-350.

\bibitem[AGM]{AGM} {\sc Artstein-Avidan S., Giannopoulos A., Milman V.D.} \textit{Asymptotic Geometric Analysis, Part 1.}
Mathematical Surveys and Monographs  {\bf 122}, (2015), American Mathematical Society, Providence, RI.

\bibitem[B]{B} {\sc Ball, K.} \textit{Logarithmically concave functions and sections of convex
sets in $\R^n$.} Studia Math. {\bf 88},(1) (1988),pp. 69–84.

\bibitem[F]{F} {\sc  Firey W. J.}
\textit{Polar Means of Convex Bodies and a Dual to the Brunn-Minkowski theorem.}
Canadian Math. J. {\bf 13} (1961), pp. 444--453.

\bibitem[F2]{F2}{\sc  Firey W. J.}
\textit{Mean cross-section measures of harmonic means of convex bodies.}
Pacific J. Math. {\bf 11} (1961), pp. 1263--1266.


\bibitem[HY]{HY} {\sc Hern\'andez Cifre M.A., Yepes Nicol\'as J.} \textit{On Brunn-Minkowski type inequalities for polar bodies.}
Journal of geometric Analysis {\bf 26}, no.1 (2016), pp. 143--155.


\bibitem[MP]{MP} {\sc Milman V.D., Pajor A.}
\textit{Entropy and Asymptotic Geometry of Non-Symmetric Convex Bodies.}
Advances in Mathematics {\bf 152} n.2, (2000), pp. 314--335.

\bibitem[RS]{RS}{\sc Rogers C. A., Shephard G. C.}
\textit{The difference body of a convex body.}  Arch. Math. {\bf 8} (1957), pp. 220--233

\bibitem[RS2]{RS2}{\sc Rogers C. A., Shephard G. C.}
\textit{Convex bodies associated with a given convex body.} J. Lond. Math. Soc. {\bf 33} (1958), pp. 270--281.

\end{thebibliography}
\end{document}